\documentclass[reqno]{amsart}
\usepackage{amssymb}
\usepackage{ytableau}
\usepackage{hyperref}
\usepackage[capitalise]{cleveref}
\usepackage{thm-restate}
\allowdisplaybreaks

\newtheorem{thm}{Theorem}[section]
\newtheorem{prop}[thm]{Proposition}
\newtheorem{lem}[thm]{Lemma}
\newtheorem{cor}[thm]{Corollary}
\theoremstyle{definition}
\newtheorem{rem}[thm]{Remark}
\newtheorem{nota}[thm]{Notation}
\newtheorem{exa}[thm]{Example}
\makeatletter
\newcommand{\neutralize}[1]{\expandafter\let\csname c@#1\endcsname\count@}
\makeatother
\newenvironment{thmp}[1]
    {\renewcommand{\thethm}{\ref*{#1}$'$}
        \neutralize{thm}\phantomsection
        \begin{thm}}
    {\end{thm}}
\newenvironment{propp}[1]
    {\renewcommand{\theprop}{\ref*{#1}$'$}
        \neutralize{prop}\phantomsection
        \begin{prop}}
    {\end{prop}}

\newcommand{\bC}{\mathbb{C}}
\newcommand{\bE}{\mathbb{E}}

\newcommand{\bfZ}{\mathbf{Z}}
\newcommand{\bfi}{\mathbf{i}}
\newcommand{\bfj}{\mathbf{j}}
\newcommand{\bfp}{\mathbf{p}}
\newcommand{\bfx}{\mathbf{x}}

\newcommand{\C}{\mathrm{C}}
\newcommand{\Comp}{\mathrm{Comp}}
\newcommand{\End}{\mathrm{End}}
\newcommand{\GL}{\mathrm{GL}}
\newcommand{\Hom}{\mathrm{Hom}}
\newcommand{\Imm}{\mathrm{Imm}}
\newcommand{\Ind}{\mathrm{Ind}}
\newcommand{\Orb}{\mathrm{Orb}}
\newcommand{\Res}{\mathrm{Res}}
\newcommand{\SST}{\mathrm{SST}}
\newcommand{\Stab}{\mathrm{Stab}}
\newcommand{\Sym}{\mathrm{Sym}}
\newcommand{\Tr}{\mathrm{Tr}}
\newcommand{\WComp}{\mathrm{WComp}}
\newcommand{\Wg}{\mathrm{Wg}}
\newcommand{\diag}{\mathrm{diag}}
\newcommand{\mult}{\mathrm{mult}}
\newcommand{\sgn}{\mathrm{sgn}}

\newcommand{\triv}{\mathrm{triv}}

\newcommand{\sfM}{\mathsf{M}}
\newcommand{\sfe}{\mathsf{e}}
\newcommand{\sfm}{\mathsf{m}}
\newcommand{\sfs}{\mathsf{s}}

\newcommand{\la}{\langle}
\newcommand{\ra}{\rangle}

\author{Jacob Campbell}
\address{Department of Pure Mathematics, University of Waterloo, 200 University Avenue West, Waterloo, Ontario N2L 3G1, Canada}
\email{j48campb@uwaterloo.ca}

\title{Commutators in finite free probability, I}

\begin{document}

\maketitle

\begin{abstract}
    This paper describes the expected characteristic polynomial of the commutator of randomly rotated matrices, in the context of the finite free probability theory initiated by Marcus, Spielman, and Srivastava. The key technical features are the use of Weingarten calculus to translate the random matrix problem into one of combinatorial representation theory, followed by some applications of the Goulden-Jackson immanant formula and the classical theory of permutation modules.
\end{abstract}

\section{Introduction}

In recent years, the pioneering work \cite{MSS22} of Marcus, Spielman, and Srivastava has spawned a theory of \emph{finite free probability}, which is concerned with random unitary rotations of $d \times d$ matrices, whose behaviour at the level of roots and characteristic polynomials parallels (and converges to) free probability. Marcus \cite{M21} has worked out an analytic approach, with finite $R$- and $S$-transforms that parallel Voiculescu's original development of free probability in the 1980s. A combinatorial approach, with finite free cumulants, has been initiated in \cite{AP18,AGP21}, in parallel with the work of Nica and Speicher in the 1990s.

The starting point is the following pair of operations on polynomials:

\begin{nota}
    \label{nota:PolyConv}
    Let $p(x)$ and $q(x)$ be monic polynomials with degree $d$, say
    \[ p(x) = \sum_{k=0}^d x^{d-k} (-1)^k a_k \text{ and } q(x) = \sum_{k=0}^d x^{d-k} (-1)^k b_k \text{.} \]
    Then define
    \[ p(x) \boxplus_d q(x) := \sum_{k=0}^d x^{d-k} (-1)^k \left( \sum_{i+j=k} \frac{(d-i)! (d-j)!}{d! (d-k)!} a_i b_j \right) \]
    and
    \[ p(x) \boxtimes_d q(x) := \sum_{k=0}^d x^{d-k} (-1)^k \left( \frac{1}{\binom{d}{k}} a_k b_k \right) \text{.} \]
\end{nota}

As mentioned in \cite{MSS22}, these operations appeared in the literature on roots of polynomials in the 1920s. For the purposes of this paper, however, their meaning is tied to the following random matrix interpretation:

\begin{thm}[{\cite{MSS22}}]
    \label{thm:MSS}
    Pick $d \times d$ normal matrices $A$ and $B$ which have characteristic polynomials $p(x)$ and $q(x)$ respectively. Then
    \[ p(x) \boxplus_d q(x) = \bE_U c_x(A + U B U^*) \text{ and } p(x) \boxtimes_d q(x) = \bE_U c_x(A U B U^*) \]
    where $U$ is a $d \times d$ random unitary matrix, and $c_x(\cdot)$ is the characteristic polynomial.
\end{thm}

After the operations of addition and multiplication, the next natural question concerns the commutator: what is the expected characteristic polynomial of
\[ A U B U^* - U B U^* A \]
for a $d \times d$ random unitary matrix $U$? In this paper, this question is answered in terms of $\boxplus_d$ and $\boxtimes_d$.

\begin{nota}
    In the setup of \cref{nota:PolyConv}, write
    \begin{align*}
        p(x) \boxminus_d q(x) :&= \sum_{k=0}^d x^{d-k} (-1)^k \left( \sum_{i+j=k} (-1)^j \frac{(d-i)! (d-j)!}{d! (d-k)!} a_i b_j \right) \\
        &= \bE_U c_x(A - U B U^*)
    \end{align*}
    for the operation of ``subtraction" with respect to $\boxplus_d$.
\end{nota}

\begin{nota}
    Write
    \[ z_d(x) = \sum_{k=0}^{\lfloor d/2 \rfloor} x^{d-2k} \binom{d}{2k} (d)_k \frac{k!}{(2k)!} \frac{d+1-k}{d+1} \]
    for $d \geq 1$.
\end{nota}

\begin{thm}
    \label{thm:Main}
    Let $A$ and $B$ be $d \times d$ normal matrices with characteristic polynomials $p(x)$ and $q(x)$ respectively. Then
    \[ \bE_U c_x(A U B U^* - U B U^* A) = (p(x) \boxminus_d p(x)) \boxtimes_d (q(x) \boxminus_d q(x)) \boxtimes_d z_d(x) \]
    where $U$ is a $d \times d$ random unitary matrix.
\end{thm}

In the literature on finite free probability, is it typical to work with the elementary symmetric functions in the eigenvalues of a matrix, i.e. the coefficients of the characteristic polynomial:

\begin{nota}
    As mentioned above, $c_x(X)$ is the characteristic polynomial of a $d \times d$ matrix $X$, in the variable $x$. The coefficients are
    \[ c_x(X) = \sum_{k=0}^d x^{d-k} (-1)^k \sfe_k(X) \]
    where $\sfe_k(X)$ is the $k$-th elementary symmetric function (reviewed in \cref{nota:Sym}) in the eigenvalues of $X$.
\end{nota}

In terms of the coefficients, the main result of this paper is

\begin{thmp}{thm:Main}
    For $0 \leq k \leq d$, we have
    \begin{align*}
        \bE_U \sfe_k(A U B U^* - U B U^* A) &= \left( \sum_{i+j=k} (-1)^i \frac{(d-i)! (d-j)!}{d! (d-k)!} \sfe_i(A) \sfe_j(A) \right) \\
        &\qquad \left( \sum_{i+j=k} (-1)^i \frac{(d-i)! (d-j)!}{d! (d-k)!} \sfe_i(B) \sfe_j(B) \right) \\
        &\qquad \frac{(d-k)!}{(d-k/2)!} (k/2)! \frac{d+1-k/2}{d+1}
    \end{align*}
    if $k$ is even, and $0$ if $k$ is odd.
\end{thmp}

Beginning in a manner similar to \cite{CZ21}, one can use Weingarten calculus to reduce the analytic problem of computing expected symmetric functions in eigenvalues of random matrices to a problem of combinatorial representation theory. After some work on the representation theory side, one lands on a particular immanant which is non-trivial but tractable by a 1992 result of Goulden and Jackson, which relates immanants to Schur polynomials. Then, the proof of \cref{thm:Main} amounts to symmetric function computations.

It is known by now \cite{M21,AP18,AGP21} that the operations $\boxplus_d$ and $\boxtimes_d$ converge, respectively, to free additive and multiplicative convolution as $d \to \infty$. An important demonstration of the power of free probability theory -- specifically, of free cumulants -- was the description of the commutator of free random variables by Nica and Speicher in \cite{NS98}. (It should be noted that \cref{thm:Main} is quite reminiscent of their result, when the latter is phrased in terms of $R$-transforms.) The argument in their paper is combinatorial, using free cumulants and the notion of $R$-diagonality, and ultimately coming down to large non-trivial combinatorial cancellations. One must ask, then, how \cref{thm:Main} connects with finite free cumulants for finite $d$, and with free probability as $d \to \infty$. These questions will be taken up in a future paper \cite{C}.

Apart from this introduction and a section of preliminaries, the paper is organized as follows: \cref{sec:Immanant} is a self-contained computation of the immanants of a class of low-rank matrices, to be used later. \cref{sec:ExpSym} begins with a straightforward application of Weingarten calculus to manipulate the expected elementary symmetric functions of the commutator into a form which suggests that the dependence on $A$ is only through the immanant of a certain low-rank matrix derived from its eigenvalues. To substantiate this, some work is needed on the representation-theoretic description of the Weingarten function; this is more or less textbook material in the representation theory of finite groups and is done in \cref{sec:Young}. The remainder of \cref{sec:ExpSym} completes the proof of \cref{thm:Main}, using the results of \cref{sec:Immanant} and \cref{sec:Young}, taking for granted some rather heavy symmetric function computations which are done in \cref{sec:Transition}.

\section{Preliminaries and notation}

\begin{nota}[Integer and set partitions]
    The notation $\lambda \vdash k$ means that $\lambda$ is a partition of $k$, or equivalently a Young diagram with $k$ boxes; $\lambda^T$ is the transpose of $\lambda$. Write $\rho^{\lambda} : S_k \to \GL(V^{\lambda})$ for the irreducible representation of $S_k$ labeled by $\lambda$, and the same for the linear extension $\rho^{\lambda} : \bC[S_k] \to \End(V^{\lambda})$. The character of $\rho^{\lambda}$ will be written as $\chi^{\lambda}$. (A good reference on the representation theory of $S_k$ is \cite{CST-book}.)

    There is a particular type of partition for which we will make special notation: for $0 \leq p \leq \lfloor k/2 \rfloor$, write
    \[ 2_k^p := (\underbrace{2,\ldots,2}_p,\underbrace{1,\ldots,1}_{k-2p}) \]
    which is a partition of $k$ with length $k-p$.
    
    The relation $\trianglelefteq$ on partitions of $k$, called the \emph{dominance ordering}, is defined by $\mu \trianglelefteq \lambda$ when $\sum_{i=1}^m \mu_i \leq \sum_{i=1}^m \lambda_i$ for all $m \geq 1$.
    
    Write $P(k)$ for the set of partitions of the set $[k] := \{ 1,\ldots,k \}$. For $\pi \in P(k)$, we will use the notation $t(\pi)$ for the \emph{integer} partition of $k$ whose parts are the sizes of the blocks in $\pi$, in non-increasing order.
\end{nota}

\begin{exa}
    The partitions of $6$ which are ``below" $(2,2,2)$ in the dominance ordering are $(1,1,1,1,1,1)$, $(2,1,1,1,1)$, $(2,2,1,1)$, and $(2,2,2)$.
\end{exa}

\begin{nota}[Integer compositions]
    Write $\Comp(k)$ for the set of compositions of $k$, and $\WComp(k)$ for the set of weak compositions of $k$. For $I \in \Comp(k)$, write
    \[ S_I := S_{I_1} \times \cdots \times S_{I_l} \text{,} \]
    viewed as a subgroup of $S_k$ in the obvious way. This is called a \emph{Young subgroup}.
\end{nota}

\subsection{Symmetric functions}

We will use the language of symmetric functions extensively; a comprehensive reference is \cite{M-book}, as well as \cite{CST-book} for particular aspects of the theory.

\begin{nota}
    \label{nota:Sym}
    We use lower-case sans-serif letters for symmetric functions. To define the \emph{monomial} symmetric functions, consider $\lambda \vdash k$ as an infinite sequence $(\lambda_1,\ldots,\lambda_l,0,0,\ldots)$ and let $S_{\infty}$ act in the obvious way. Then we write
    \[ \sfm_{\lambda}(\bfx) = \sum_{I \in \Orb(\lambda)} x_1^{I_1} x_2^{I_2} \cdots \]
    which makes sense because $I$ has finitely many non-zero entries. The special cases $\sfm_{(1^k)}$, where $(1^k) := (1,\ldots,1) \vdash k$, are denoted by $\sfe_k$, and these are called the \emph{elementary} symmetric functions. The symmetric functions $\sfe_{\lambda}$, for $\lambda \vdash k$, are defined multiplicatively.
\end{nota}

\begin{exa}
    Let $\lambda = (2,2,1) \vdash 5$. Then
    \[ \sfm_{\lambda}(x_1,x_2,x_3) = x_1^2 x_2^2 x_3 + x_1^2 x_2 x_3^2 + x_1 x_2^2 x_3^2 \]
    and
    \begin{align*}
        \sfe_{\lambda}(x_1,x_2,x_3) &= (x_1 x_2 + x_1 x_3 + x_2 x_3)^2 (x_1 + x_2 + x_3) \\
        &= (x_1^2 x_2^2 + x_1^2 x_3^2 + x_2^2 x_3^2 + 2 x_1^2 x_2 x_3 + 2 x_1 x_2^2 x_3 + 2 x_1 x_2 x_3^2) \\
        &\qquad (x_1 + x_2 + x_3) \\
        &= x_1^3 x_2^2 + x_1^2 x_2^3 + x_1^3 x_3^2 + x_1^2 x_3^3 + x_2^3 x_3^2 + x_2^2 x_3^3 \\
        &\qquad + 2 x_1^3 x_2 x_3 + 2 x_1 x_2^3 x_3 + 2 x_1 x_2 x_3^3  \\
        &\qquad + 5 x_1^2 x_2^2 x_3 + 5 x_1^2 x_2 x_3^2 + 5 x_1 x_2^2 x_3^2 \text{.}
    \end{align*}
    An attentive reader might notice that the latter is
    \[ \sfm_{(3,2)}(x_1,x_2,x_3) + 2\sfm_{(3,1,1)}(x_1,x_2,x_3) + 5 \sfm_{(2,2,1)}(x_1,x_2,x_3) \]
    which is a special case of the more general relationship between the elementary and monomial bases of the algebra of symmetric functions. This will be explained and used in \cref{sec:Transition}.
\end{exa}

\begin{nota}[Schur polynomials]
    For $\lambda \vdash k$, let $\SST(\lambda)$ be the set of semistandard Young tableaux, i.e. the Young tableaux whose rows are non-decreasing and whose columns are strictly increasing. For $T \in \SST(\lambda)$, the \emph{weight} of $T$, denoted by $\omega(T) = (\omega_1(T),\ldots,\omega_k(T))$, is defined by letting $\omega_i(T)$ be the number of $i$s in $T$. The polynomial
    \[ \sfs_{\lambda}(x_1,\ldots,x_k) := \sum_{T \in \SST(\lambda)} x_1^{\omega_1(T)} \cdots x_k^{\omega_k(T)} \]
    is called the \emph{Schur polynomial} labeled by $\lambda$.
\end{nota}

\begin{lem}
    \label{lem:Hooks}
    We have
    \[ \dim(2_k^p) = \frac{k! (k-2p+1)}{p! (k-p+1)!} \]
    and
    \[ \sfs_{2_k^p}(1^d) = \frac{k-2p+1}{p! (k-p+1)!} (d+1)_p (d)_{k-p} \]
    for $0 \leq p \leq \lfloor k/2 \rfloor$.
\end{lem}
\begin{proof}
    These follow directly from the well-known formulae
    \[ \sfs_{\lambda}(1^d) = \frac{\dim(\lambda)}{k!} \prod_{(i,j) \in \lambda} (d+c_{\lambda}(i,j)) \text{ and } \dim(\lambda) = \frac{k!}{\prod_{(i,j) \in \lambda} h_{\lambda}(i,j)} \text{,} \]
    see e.g. \cite[Theorems 4.3.3 \& 4.2.14]{CST-book} respectively. The hooks and contents are illustrated in \cref{fig:ContentsHooks}.
\end{proof}

\begin{figure}
    \centering
    \begin{tabular}{c@{\hspace{5em}}c}
        contents & hooks \\
        \ytableausetup{boxsize=3.2em} \begin{ytableau}
            0 & 1 \\
            -1 & 0 \\
            \none[\vdots] & \none[\vdots] \\
            \scriptstyle -(p-2) & \scriptstyle -(p-3) \\
            \scriptstyle -(p-1) & \scriptstyle -(p-2) \\
            -p \\
            \scriptstyle -(p+1) \\
            \none[\vdots] \\
            \scriptscriptstyle -(k-p-2) \\
            \scriptscriptstyle -(k-p-1)
        \end{ytableau} & \begin{ytableau}
            \scriptstyle k-p+1 & p \\
            \scriptstyle k-p & p-1 \\
            \none[\vdots] & \none[\vdots] \\
            \scriptstyle k-2p+3 & 2 \\
            \scriptstyle k-2p+2 & 1 \\
            k-2p \\
            \scriptstyle k-2p-1 \\
            \none[\vdots] \\
            2 \\
            1
        \end{ytableau}
    \end{tabular}
    \caption{Contents and hooks for $2_k^p$} \label{fig:ContentsHooks}
\end{figure}

\subsection{Weingarten calculus}
\label{ssec:Weingarten}

On the random matrix side, our main tool is Weingarten calculus, which reduces integration of polynomial functions on certain compact matrix groups to combinatorial representation theory. The first piece is the following integration formula, in terms of a particular sequence of class functions $(\Wg_{k,d}^U)_{k \geq 1}$ on $S_k$:

\begin{thm}[\cite{C03,CS06}]
    \label{thm:UnitaryWeingarten}
    For $k,k' \geq 1$ and $\bfi,\bfj : [k] \to [d]$ and $\bfi',\bfj' : [k'] \to [d]$, the integral
    \[ \int_{U_d} u_{\bfi(1) \bfj(1)} \cdots u_{\bfi(k) \bfj(k)} \overline{u_{\bfi'(1) \bfj'(1)}} \cdots \overline{u_{\bfi'(k') \bfj'(k')}} \, dU \]
    is
    \[ \sum_{\substack{\pi,\sigma \in S_k \\ \bfi = \bfi' \circ \pi \\ \bfj = \bfj' \circ \sigma}} \Wg_{k,d}^U(\pi^{-1} \sigma) \]
    when $k = k'$, and it is $0$ otherwise.
\end{thm}

Also important for our purposes is the following description of $\Wg^U$ in terms of the representation theory of $S_k$:

\begin{thm}[{\cite[Proposition 2.3]{CS06}}]
    \label{thm:SchurWeingarten}
    We have
    \[ \Wg_{k,d}^U = \frac{1}{(k!)^2} \sum_{\substack{\lambda \vdash k \\ \ell(\lambda) \leq d}} \frac{\dim(\lambda)^2}{\sfs_{\lambda}(1^d)} \chi^{\lambda} \]
    for $k \geq 1$.
\end{thm}

\section{Immanants of a rank-two matrix}
\label{sec:Immanant}

This self-contained section is dedicated to the description of the immanant of a certain type of rank $2$ matrix, in terms of symmetric functions in its eigenvalues.

\begin{nota}
    For $\lambda \vdash k$, the \emph{immanant} of a $k \times k$ matrix $X$ with respect to $\lambda$ is
    \[ \Imm^{\lambda}(X) := \sum_{\sigma \in S_k} \chi^{\lambda}(\sigma) \prod_{i=1}^k x_{i \sigma(i)} \text{.} \]
    This is a common generalization of the determinant and permanent, which are the cases $\lambda = (1^k)$ and $\lambda = (k)$ respectively.
\end{nota}

\begin{nota}
    \label{nota:DiffMatrix}
    For a $k \times k$ diagonal matrix $X = \diag(x_1,\ldots,x_k)$, write $\delta_{\pm}(X) := (x_i \pm x_j)_{i,j}$.
\end{nota}

\begin{thm}
    \label{prop:Immanant}
    Let $X$ be a $k \times k$ diagonal matrix. Then
    \[ \Imm^{\lambda}(\delta_-(X)) = \begin{cases}
        (-1)^{\lambda_2} \sum_{l=0}^k (-1)^l (k-l)! l! \sfe_{k-l}(X) \sfe_l(X) & \text{if } \ell(\lambda) \leq 2 \\
        0 & \text{otherwise}
    \end{cases} \]
    for $\lambda \vdash k$.
\end{thm}

The main technical tool for the proof of \cref{prop:Immanant} is the following result of Goulden and Jackson:

\begin{prop}[{\cite[Equation (9)]{GJ92}}]
    \label{prop:GouldenJackson}
    Let $Y$ be a $k \times k$ matrix, let $\lambda \vdash k$, let $z_1,\ldots,z_k$ be formal commuting variables, write $Z := \diag(z_1,\ldots,z_k)$, and let $\alpha_1,\ldots,\alpha_k$ be the eigenvalues of $ZY$. Then $\Imm^{\lambda}(Y)$ is the coefficient of $z_1 \cdots z_k$ in $\sfs_{\lambda}(\alpha_1,\ldots,\alpha_k)$.
\end{prop}

\begin{lem}
    \label{lem:DiffEigen}
    Let $k \geq 1$, $X = \diag(x_1,\ldots,x_k)$, and $Z = \diag(z_1,\ldots,z_k)$. Then the characteristic polynomial of $Z \delta_-(X)$ is
    \[ \det(xI - Z \delta_-(X)) = x^k + \left( \sum_{1 \leq i < j \leq k} z_i z_j (x_i - x_j)^2 \right) x^{k-2} \]
    whose roots are
    \[ \pm i \sqrt{\sum_{i < j} z_i z_j (x_i-x_j)^2} \]
    with multiplicity $1$ each and $0$ with multiplicity $k-2$.
\end{lem}

\begin{rem}
    The characteristic polynomial of $Z \delta_+(X)$ is somewhat more complicated; the rank is still $2$, so all but two of the eigenvalues are $0$, but the non-zero ones differ from each other in a less trivial way. One might compare this with the difference in tractability between the commutator and anti-commutator observed in \cite{NS98}.
\end{rem}

\begin{lem}
    \label{lem:RankTwoSchur}
    We have
    \[ \sfs_{\lambda}(\alpha,\beta,0,\ldots,0) = \begin{cases}
        \alpha^k \frac{\left( \frac{\beta}{\alpha} \right)^{\lambda_2} - \left( \frac{\beta}{\alpha} \right)^{\lambda_1+1}}{1 - \frac{\beta}{\alpha}} & \text{if } \ell(\lambda) \leq 2 \\
        0 & \text{otherwise}
    \end{cases} \]
    for $\lambda \vdash k$.
\end{lem}
\begin{proof}
    If $\ell(\lambda) > 2$, then every semistandard tableau of shape $\lambda$ has $\omega_i(T) > 0$ for some $i > 2$, so
    \[ \sfs_{\lambda}(\alpha,\beta,0,\ldots,0) = \sum_{T \in \SST(\lambda)} \alpha^{\omega_1(T)} \beta^{\omega_2(T)} 0^{\omega_3(T)} \cdots = 0 \text{.} \]
    On the other hand, if $\ell(\lambda) \leq 2$, the only semistandard tableaux of shape $\lambda$ with $\omega_i(T) = 0$ for all $i > 2$ are of the form
    \[ \ytableausetup{boxsize=2em} \begin{ytableau}
        1 & \none[\cdots] & 1 & 1 & \none[\cdots] & 1 & 2 & \none[\cdots] & 2 \\
        2 & \none[\cdots] & 2
    \end{ytableau} \]
    where the first row has $0 \leq t \leq \lambda_1-\lambda_2$ boxes with $2$s. So
    \begin{align*}
        \sfs_{\lambda}(\alpha,\beta,0,\ldots,0) &= \sum_{t=0}^{\lambda_1-\lambda_2} \alpha^{k-(t+\lambda_2)} \beta^{t+\lambda_2} \\
        &= \alpha^{k-\lambda_2} \beta^{\lambda_2} \sum_{t=0}^{\lambda_1-\lambda_2} \left( \frac{\beta}{\alpha} \right)^t \\
        &= \alpha^k \left( \frac{\beta}{\alpha} \right)^{\lambda_2} \frac{1 - \left( \frac{\beta}{\alpha} \right)^{\lambda_1 - \lambda_2 + 1}}{1 - \frac{\beta}{\alpha}}
    \end{align*}
    which is the non-zero expression in the claim.
\end{proof}

\begin{proof}[Proof of \cref{prop:Immanant}]
    By \cref{prop:GouldenJackson}, \cref{lem:RankTwoSchur}, and \cref{lem:DiffEigen},
    \[ \Imm^{\lambda}(\delta_-(X)) \]
    is the coefficient of $z_1 \cdots z_k$ in
    \[ \sfs_{\lambda}(\alpha,-\alpha,0,\ldots,0) = \begin{cases}
        (-1)^{\lambda_2} \alpha^k & \text{if } \ell(\lambda) \leq 2 \\
        0 & \text{otherwise}
    \end{cases} \]
    where $\pm \alpha$ are the non-zero eigenvalues of $\delta_-(X)$. But this expression does not depend on $\lambda$ except for a straightforward sign, so we can simply take $\lambda = (k)$ and directly compute the common quantity for the cases $\ell(\lambda) \leq 2$, which is the permanent of $\delta_-(X)$:
    \begin{align*}
        &\quad \Imm^{(k)}(\delta_-(X)) \\
        &= \sum_{\sigma \in S_k} \prod_{i=1}^k (x_i-x_{\sigma(i)}) \\
        &= \sum_{\sigma \in S_k} \sum_{R \subseteq [k]} (-1)^{| R |} \prod_{i \in R} x_i \prod_{i \notin R} x_{\sigma(i)} \\
        &= \sum_{R \subseteq [k]} (-1)^{| R |} \det(X(R,R)) \left( \sum_{\substack{\sigma : [k] \to [k] \\ \text{injective}}} \prod_{i \in [k] \setminus R} x_{\sigma(i)} \right) \\
        &= \sum_{R \subseteq [k]} (-1)^{| R |} \det(X(R,R)) | R |! \left( \sum_{\substack{\sigma : [k] \setminus R \to [k] \\ \text{injective}}} \prod_{i \in [k] \setminus R} x_{\sigma(i)} \right) \\
        &= \sum_{R \subseteq [k]} (-1)^{| R |} \det(X(R,R)) | R |! (k-| R |)! \left( \sum_{\substack{\sigma : [k] \setminus R \to [k] \\ \text{increasing}}} \prod_{i \in [k] \setminus R} x_{\sigma(i)} \right) \\
        &= \sum_{l=0}^k (-1)^l \sum_{\substack{R \subseteq [k] \\ | R | = l}} \det(X(R,R)) (k-l)! l! \left( \sum_{1 \leq i_1 < \cdots < i_{k-l} \leq k} x_{i_1} \cdots x_{i_{k-l}} \right) \\
        &= \sum_{l=0}^k (-1)^l (k-l)! l! \sfe_{k-l}(X) \sfe_l(X) \text{.}
    \end{align*}
\end{proof}

\section{Expected symmetric functions}
\label{sec:ExpSym}

\begin{nota}
    In this section, fix $d \times d$ normal matrices $A$ and $B$. Observe that if $A$ and $B$ are diagonalized by unitaries as $A = V_A D_A V_A^*$ and $B = V_B D_B V_B^*$ respectively, then since $c_x(\cdot)$ is invariant under unitary conjugation,
    \begin{align*}
        &\quad \bE_U c_x(A U B U^* - U B U^* A) \\
        &= \bE_U c_x((V_A D_A V_A^*) U (V_B D_B V_B^*) U^* - U (V_B D_B V_B^*) U^* (V_A D_A V_A^*)) \\
        &= \bE_U c_x(D_A (V_A^* U V_B) D_B (V_A^* U V_B)^* - (V_A^* U V_B) D_B (V_A^* U V_B)^* D_A) \text{.}
    \end{align*}
    Due to the invariance of the Haar measure on the group of $d \times d$ unitary matrices, the above is just
    \[ \bE_U c_x(D_A U D_B U^* - U D_B U^* D_A) \]
    so we can assume without loss of generality that $A$ and $B$ are diagonal, say $A = \diag(a_1,\ldots,a_d)$ and $B = \diag(b_1,\ldots,b_d)$.
\end{nota}

Recall the main result, in terms of coefficients:

\begin{thmp}{thm:Main}
    For $0 \leq k \leq d$, we have
    \begin{align*}
        \bE_U \sfe_k(A U B U^* - U B U^* A) &= \left( \sum_{i+j=k} (-1)^i \frac{(d-i)! (d-j)!}{d! (d-k)!} \sfe_i(A) \sfe_j(A) \right) \\
        &\qquad \left( \sum_{i+j=k} (-1)^i \frac{(d-i)! (d-j)!}{d! (d-k)!} \sfe_i(B) \sfe_j(B) \right) \\
        &\qquad \frac{(d-k)!}{(d-k/2)!} (k/2)! \frac{d+1-k/2}{d+1}
    \end{align*}
    if $k$ is even, and $0$ if $k$ is odd.
\end{thmp}

In a manner similar to \cite{CZ21}, one can proceed very directly to untangle the elementary symmetric function in terms of the entries of the matrix:

\begin{lem}
    \label{lem:StartingPoint}
    We have
    \begin{align*}
        &\quad \bE_U \sfe_k(A U B U^* \pm U B U^* A) \\
        &= \frac{1}{(k!)^2} \sum_{\lambda \vdash k} \frac{\dim(\lambda)^2}{\sfs_{\lambda}(1^d)} \sum_{\substack{S \subseteq [d] \\ | S | = k}} \sum_{\bfp : S \to [d]} \left( \prod_{i \in S} b_{\bfp(i)} \right) \\
        &\qquad \sum_{\sigma \in \Sym(S)} \sgn(\sigma) \left( \sum_{\substack{\tau \in \Sym(S) \\ \bfp = \bfp \circ \tau}} \chi^{\lambda}(\sigma \tau) \right) \prod_{i \in S} (a_i \pm a_{\sigma(i)})
    \end{align*}
    for $0 \leq k \leq d$.
\end{lem}
\begin{proof}
     The $(i,j)$-th entry of $A U B U^* \pm U B U^* A$ is
    \[ \sum_{p=1}^d a_i u_{ip} b_p \overline{u_{jp}} \pm \sum_{p=1}^d u_{ip} b_p \overline{u_{jp}} a_j \]
    so we have
    \begin{align}
        &\quad \sfe_k(A U B U^* \pm U B U^* A) \nonumber \\
        &= \sum_{\substack{S \subseteq [d] \\ | S | = k}} \sum_{\sigma \in \Sym(S)} \sgn(\sigma) \prod_{i \in S} \sum_{p=1}^d (a_i u_{i p} b_p \overline{u_{\sigma(i) p}} \pm u_{i p} b_p \overline{u_{\sigma(i) p}} a_{\sigma(i)}) \nonumber \\
        &= \sum_{\substack{S \subseteq [d] \\ | S | = k}} \sum_{\sigma \in \Sym(S)} \sgn(\sigma) \sum_{\bfp : S \to [d]} \nonumber \\
        &\qquad \prod_{i \in S} (a_i u_{i \bfp(i)} b_{\bfp(i)} \overline{u_{\sigma(i) \bfp(i)}} \pm u_{i \bfp(i)} b_{\bfp(i)} \overline{u_{\sigma(i) \bfp(i)}} a_{\sigma(i)}) \nonumber \\
        &= \sum_{\substack{S \subseteq [d] \\ | S | = k}} \sum_{\sigma \in \Sym(S)} \sgn(\sigma) \sum_{\bfp : S \to [d]} \prod_{i \in S} (a_i \pm a_{\sigma(i)}) u_{i \bfp(i)} b_{\bfp(i)} \overline{u_{\sigma(i) \bfp(i)}} \nonumber \\
        &= \sum_{\substack{S \subseteq [d] \\ | S | = k}} \sum_{\bfp : S \to [d]} \left( \prod_{i \in S} b_{\bfp(i)} \right) \nonumber \\
        &\qquad \sum_{\sigma \in \Sym(S)} \sgn(\sigma) \prod_{i \in S} (a_i \pm a_{\sigma(i)}) \prod_{i \in S} u_{i \bfp(i)} \overline{u_{\sigma(i) \bfp(i)}} \label{eq:noexp}
    \end{align}   
    and the claim amounts to a straightforward application of \cref{thm:UnitaryWeingarten} and \cref{thm:SchurWeingarten}, with the observation that $\ell(\lambda) \leq k \leq d$ for all $\lambda \vdash k$. Namely, we have
    \begin{align*}
        \bE_U \left( \prod_{i \in S} u_{i \bfp(i)} \overline{u_{\sigma(i) \bfp(i)}} \right) &= \sum_{\substack{\pi,\tau \in \Sym(S) \\ 1 = \sigma \circ \pi \\ \bfp = \bfp \circ \tau}} \Wg_{k,d}^U(\pi^{-1} \tau) \tag{\cref{thm:UnitaryWeingarten}} \\
        &= \sum_{\substack{\tau \in \Sym(S) \\ \bfp = \bfp \circ \tau}} \Wg_{k,d}^U(\sigma \tau) \\
        &= \sum_{\substack{\tau \in \Sym(S) \\ \bfp = \bfp \circ \tau}} \frac{1}{(k!)^2} \sum_{\lambda \vdash k} \frac{\dim(\lambda)^2}{\sfs_{\lambda}(1^d)} \chi^{\lambda}(\sigma \tau) \tag{\cref{thm:SchurWeingarten}, $\ell(\lambda) \leq k \leq d$ for all $\lambda \vdash k$}
    \end{align*}
    so the expectation of \cref{eq:noexp} is
    \begin{align*}
        &\quad \sum_{\substack{S \subseteq [d] \\ | S | = k}} \sum_{\bfp : S \to [d]} \left( \prod_{i \in S} b_{\bfp(i)} \right) \sum_{\sigma \in \Sym(S)} \sgn(\sigma) \prod_{i \in S} (a_i \pm a_{\sigma(i)}) \\
        &\qquad \sum_{\substack{\tau \in \Sym(S) \\ \bfp = \bfp \circ \tau}} \frac{1}{(k!)^2} \sum_{\lambda \vdash k} \frac{\dim(\lambda)^2}{\sfs_{\lambda}(1^d)} \chi^{\lambda}(\sigma \tau) \\
        &= \frac{1}{(k!)^2} \sum_{\lambda \vdash k} \frac{\dim(\lambda)^2}{\sfs_{\lambda}(1^d)} \sum_{\substack{S \subseteq [d] \\ | S | = k}} \sum_{\bfp : S \to [d]} \left( \prod_{i \in S} b_{\bfp(i)} \right) \\
        &\qquad \sum_{\sigma \in \Sym(S)} \sgn(\sigma) \left( \sum_{\substack{\tau \in \Sym(S) \\ \bfp = \bfp \circ \tau}} \chi^{\lambda}(\sigma \tau) \right) \prod_{i \in S} (a_i \pm a_{\sigma(i)})
    \end{align*}
    hence the claim.
\end{proof}

\begin{rem}
    When $\bfp$ is injective, and for the sake of clarity we take $S = \{ 1,\ldots,k \}$, we have
    \begin{align*}
        &\quad \sum_{\sigma \in \Sym(S)} \sgn(\sigma) \left( \sum_{\substack{\tau \in \Sym(S) \\ \bfp = \bfp \circ \tau}} \chi^{\lambda}(\sigma \tau) \right) \prod_{i \in S} (a_i \pm a_{\sigma(i)}) \\
        &= \sum_{\sigma \in S_k} \sgn(\sigma) \chi^{\lambda}(\sigma) \prod_{i=1}^k (a_i \pm a_{\sigma(i)}) \\
        &= \sum_{\sigma \in S_k} \chi^{\lambda^T}(\sigma) \prod_{i=1}^k (a_i \pm a_{\sigma(i)})
    \end{align*}
    which can be immediately recognized as the immanant $\Imm^{\lambda^T}(\delta_{\pm}(A_S))$.
\end{rem}

To separate the dependence on $A$ from the dependence on $B$, in \cref{lem:StartingPoint}, the sum over $\bfp$ can be processed as follows:
\begin{align*}
    &\quad \frac{1}{(k!)^2} \sum_{\lambda \vdash k} \frac{\dim(\lambda)^2}{\sfs_{\lambda}(1^d)} \sum_{\substack{S \subseteq [d] \\ | S | = k}} \\
    &\qquad \sum_{\bfp : S \to [d]} \sum_{\sigma \in \Sym(S)} \sgn(\sigma) \left( \sum_{\substack{\tau \in \Sym(S) \\ \bfp = \bfp \circ \tau}} \chi^{\lambda}(\sigma \tau) \right) \prod_{i \in S} (a_i \pm a_{\sigma(i)}) \\
    &\qquad  \prod_{i \in S} b_{\bfp(i)} \\
    &= \frac{1}{(k!)^2} \sum_{\substack{\lambda \vdash k \\ \ell(\lambda) \leq d}} \frac{\dim(\lambda)^2}{\sfs_{\lambda}(1^d)} \sum_{\substack{S \subseteq [d] \\ | S | = k}} \\
    &\qquad \sum_{\pi \in P(S)} \left( \sum_{\sigma \in \Sym(S)} \prod_{i \in S} (a_i \pm a_{\sigma(i)}) \left( \sgn(\sigma) \sum_{\substack{\tau \in \Sym(S) \\ \tau \leq \pi}} \chi^{\lambda}(\sigma \tau) \right) \right) \\
    &\qquad \left( \sum_{\substack{\bfp : S \to [d] \\ \ker(\bfp) = \pi}} \prod_{i \in S} b_{\bfp(i)} \right)
\end{align*}
The sum over $\bfp$ in the last line only depends on $\pi$ through the sizes of its blocks:

\begin{lem}
    \label{lem:Breduction}
    Let $\mu \vdash k$ and pick $\pi \in P(k)$ with $t(\pi) = \mu$. Then
    \[ \sum_{\substack{\bfp : [k] \to [d] \\ \ker(\bfp) = \pi}} \prod_{i=1}^k b_{\bfp(i)} = \frac{\ell(\mu)!}{|\Orb(\mu)|} \sfm_{\mu}(B) \]
    where $|\Orb(\mu)|$ is the number of distinct permutations of $\mu$.
\end{lem}
\begin{proof}
    If $\pi = \{ V_1,\ldots,V_m \}$, then
    \begin{align*}
        \sum_{\substack{\bfp : [k] \to [d] \\ \ker(\bfp) = \pi}} \prod_{i=1}^k b_{\bfp(i)} &= \sum_{\substack{\bfp : \pi \to [d] \\ \text{injective}}} \prod_{V \in \pi} b_{\bfp(V)}^{| V |} \\
        &= \sum_{\rho \in S_m} \sum_{\substack{\bfp : [m] \to [d] \\ \bfp(\rho(1)) < \cdots < \bfp(\rho(m))}} b_{\bfp(1)}^{| V_1 |} \cdots b_{\bfp(m)}^{| V_m |} \\
        &= \sum_{\rho \in S_m} \sum_{\substack{\bfp : [m] \to [d] \\ \bfp(1) < \cdots < \bfp(m)}} b_{\bfp(1)}^{| V_{\rho(1)} |} \cdots b_{\bfp(m)}^{| V_{\rho(m)} |}
    \end{align*}
    and the number of duplicate summands $b_{\bfp(1)}^{|V_{\rho(1)}|} \cdots b_{\bfp(m)}^{|V_{\rho(m)}|}$ which accumulate, for each $\bfp$, as $\rho$ varies over $S_m$, is the number of permutations in $S_m$ which fix $\mu$. So
    \begin{align*}
        | \Stab(\mu) | \sum_{I \in \Orb(\mu)} \sum_{\substack{\bfp : [m] \to [d] \\ \bfp(1) < \cdots < \bfp(m)}} b_{\bfp(1)}^{I_1} \cdots b_{\bfp(m)}^{I_m} &= \frac{m!}{|\Orb(\mu)|} \sum_{I \in \Orb(\mu)} \sfM_I(B) \\
        &= \frac{\ell(\mu)!}{|\Orb(\mu)|} \sfm_{\mu}(B)
    \end{align*}
    by the orbit-stabilizer theorem.
\end{proof}

\begin{rem}
    The case $\mu=2_k^q$ will be important later: there are $\binom{k-q}{q}$ distinct permutations of
    \[ (\underbrace{\overbrace{2,\ldots,2}^q,1,\ldots,1}_{k-q}) \]
    so the multiple in \cref{lem:Breduction} is $q! (k-2q)!$.
\end{rem}

The above makes
\begin{align}
    &\quad \bE_U \sfe_k(A U B U^* \pm U B U^* A) = \frac{1}{(k!)^2} \sum_{\substack{\lambda \vdash k \\ \ell(\lambda) \leq d}} \frac{\dim(\lambda)^2}{\sfs_{\lambda}(1^d)} \sum_{\substack{S \subseteq [d] \\ | S | = k}} \nonumber \\
    &\qquad \sum_{\mu \vdash k}  \left( \sum_{\sigma \in \Sym(S)} \prod_{i \in S} (a_i \pm a_{\sigma(i)}) \left( \sgn(\sigma) \sum_{\substack{\pi \in P(k) \\ t(\pi) = \mu}} \sum_{\substack{\tau \in \Sym(S) \\ \tau \leq \pi}} \chi^{\lambda}(\sigma \tau) \right) \right) \label{eq:secondline} \\
    &\qquad \frac{\ell(\mu)!}{|\Orb(\mu)|} \sfm_{\mu}(B) \nonumber
\end{align}
and to reach the central point of the argument, one must process the sum
\[ \sgn(\sigma) \sum_{\substack{\pi \in P(k) \\ t(\pi) = \mu}} \sum_{\substack{\tau \in S_k \\ \tau \leq \pi}} \chi^{\lambda}(\sigma \tau) \]
in a way which makes the bracketed portion of \cref{eq:secondline} into a sum of immanants. This will be done in \cref{sec:Young}:

\begin{prop}
    \label{prop:Multiple}
    For $\lambda,\mu \vdash k$, there is a constant $C_{\lambda,\mu}$ such that
    \[ \sum_{\substack{\pi \in P(k) \\ t(\pi) = \mu}} \sum_{\substack{\tau \in S_k \\ \tau \leq \pi}} \chi^{\lambda}(\sigma \tau) = C_{\lambda,\mu} \chi^{\lambda}(\sigma) \text{.} \]
    with the following properties:
    \begin{enumerate}
        \item if $\mu \not\trianglelefteq \lambda$, then $C_{\lambda,\mu} = 0$;
        \item if $\lambda = 2_k^p$ and $\mu = 2_k^q$ with $0 \leq q \leq p \leq \lfloor k/2 \rfloor$, then
            \[ C_{\lambda,\mu} = \frac{p!}{(p-q)!} \binom{k-p+1}{q} \text{.} \]
    \end{enumerate}
\end{prop}

With \cref{prop:Multiple} in hand, the bracketed portion of \cref{eq:secondline} can be realized as an immanant: it is equal to
\[ \sum_{\sigma \in \Sym(S)} C_{\lambda,\mu} \sgn(\sigma) \chi^{\lambda}(\sigma) \prod_{i \in S} (a_i \pm a_{\sigma(i)}) = C_{\lambda,\mu} \Imm^{\lambda^T}(\delta_{\pm}(A_S))_{i,j \in S} \text{.} \]
This is where it seems prudent to restrict our attention to the commutator: as mentioned in \cref{sec:Immanant}, there is an apparent gap in tractability between the relevant immanants. Recall the computation from \cref{prop:Immanant}:
\begin{align*}
    &\quad \Imm^{\lambda^T}(\delta_-(A_S)) \\
    &= \begin{cases}
        (-1)^p \sum_{l=0}^k (-1)^l (k-l)! l! \sfe_{k-l}(A_S) \sfe_l(A_S) & \text{if } \lambda = 2_k^p \text{ for some } 0 \leq p \leq \lfloor k/2 \rfloor \\
        0 & \text{otherwise}
    \end{cases} \text{.}
\end{align*}
Now, the computation of
\begin{align}
    &\quad \bE_U \sfe_k(A U B U^* - U B U^* A) \nonumber \\
    &= \frac{1}{(k!)^2} \sum_{0 \leq p \leq k/2} \frac{\dim(2_k^p)^2}{\sfs_{2_k^p}(1^d)} \nonumber \\
    &\qquad \sum_{\substack{S \subseteq [d] \\ | S | = k}} \sum_{0 \leq q \leq p} C_{2_k^p,2_k^q} \left( (-1)^p \sum_{l=0}^k (-1)^l (k-l)! l! \sfe_{k-l}(A_S) \sfe_l(A_S) \right) \nonumber \\
    &\qquad q! (k-2q)! \sfm_{2_k^q}(B) \nonumber \\
    &= \left( \sum_{l=0}^k \frac{(-1)^l}{\binom{k}{l}} \sum_{\substack{S \subseteq [d] \\ | S | = k}} \sfe_{k-l}(A_S) \sfe_l(A_S) \right) \label{eq:leftdep} \\
    &\qquad \left( \frac{1}{k!} \sum_{0 \leq p \leq k/2} (-1)^p \frac{\dim(2_k^p)^2}{\sfs_{2_k^p}(1^d)} \sum_{0 \leq q \leq p} C_{2_k^p,2_k^q} q! (k-2q)! \sfm_{2_k^q}(B) \right) \label{eq:rightdep}
\end{align}
amounts to some manipulations of symmetric functions, to be carried out in the following section:

\begin{prop}
    \label{prop:LRdep}
    If $k$ is even, then
    \begin{enumerate}
        \item the expression \emph{(\ref{eq:leftdep})} is equal to
            \[ \frac{(k/2)!}{k!} \sum_{i+j=k} (-1)^i \frac{(d-i)! (d-j)!}{(d-k)! (d-k/2)!} \sfe_i(A) \sfe_j(A) \text{,} \]
            and
        \item the expression \emph{(\ref{eq:rightdep})} is equal to
            \[ k! \frac{d+1-k/2}{(d+1)! d!} \sum_{i+j=k} (-1)^i (d-i)! (d-j)! \sfe_i(B) \sfe_j(B) \text{.} \]
    \end{enumerate}
    If $k$ is odd, then the expression \emph{(\ref{eq:leftdep})} is $0$.
\end{prop}

\begin{proof}[Proof of \cref{thm:Main}]
    All the pieces are in place by now:
    \begin{align*}
        &\quad \bE_U \sfe_k(A U B U^* - U B U^* A) \\
        &= \left( \frac{(k/2)!}{k!} \sum_{i+j=k} (-1)^i \frac{(d-i)! (d-j)!}{(d-k)! (d-k/2)!} \sfe_i(A) \sfe_j(A) \right) \\
        &\qquad \left( k! \frac{d+1-k/2}{(d+1)! d!} \sum_{i+j=k} (-1)^i (d-i)! (d-j)! \sfe_i(B) \sfe_j(B) \right) \\
        &= (k/2)! \frac{(d-k)!}{(d-k/2)!} \frac{d+1-k/2}{d+1} \left( \sum_{i+j=k} (-1)^i \frac{(d-i)! (d-j)!}{(d-k)! d!} \sfe_i(A) \sfe_j(A) \right) \\
        &\qquad \left( \sum_{i+j=k} (-1)^i \frac{(d-i)! (d-j)!}{(d-k)! d!} \sfe_i(B) \sfe_j(B) \right)
    \end{align*}
    for even $0 \leq k \leq d$. For odd $k$, we have
    \[ \bE_U \sfe_k(A U B U^* - U B U^* A) = 0 \]
    since the expression (\ref{eq:leftdep}) is equal to $0$.
\end{proof}

\section{Transitions between symmetric function bases}
\label{sec:Transition}

This section is dedicated to the proof of \cref{prop:LRdep}. First of all, the claim for odd $k$ is almost trivial: the summands in the expression \ref{eq:leftdep} cancel each other out because
\[ \frac{(-1)^l}{\binom{k}{l}} \sfe_{k-l} \sfe_l = -\frac{(-1)^{k-l}}{\binom{k}{k-l}} \sfe_l \sfe_{k-l} \]
for $0 \leq l \leq k$. So for the rest of this section, $k$ is assumed to be even.

\subsection{Kostka numbers}

\begin{nota}
    The \emph{Kostka numbers}, denoted by $K(\lambda,\mu)$ for $\lambda,\mu \vdash k$, can be defined as the number of semistandard Young tableaux with shape $\lambda$ and weight $\mu$. Of course $K(\lambda,\mu)$ is non-negative, and it is non-zero if and only if $\mu \trianglelefteq \lambda$.
    
    Since the matrix $K := (K(\lambda,\mu))_{\lambda,\mu \vdash k}$ is upper-triangular with $1$s along the diagonal, it is invertible, and $K^{-1}(\lambda,\mu)$ is the $(\lambda,\mu)$-th entry of its inverse. These so-called \emph{inverse Kostka numbers} have a nice combinatorial interpretation \cite{ER90} in terms of Young diagrams.
\end{nota}

Another important interpretation of the Kostka numbers is that they describe transitions between different bases of the symmetric functions, including the elementary and monomial bases. The general principle can be read from e.g. \cite[Section I.6]{M-book}:

\begin{prop}
    We have
    \[ \sfe_{\lambda} = \sum_{\mu \vdash k} \left( \sum_{\nu \vdash k} K(\nu,\lambda) K(\nu^T,\mu) \right) \sfm_{\mu} \]
    and
    \[ \sfm_{\lambda} = \sum_{\mu \vdash k} \left( \sum_{\nu \vdash k} K^{-1}(\lambda,\nu^T) K^{-1}(\mu,\nu) \right) \sfe_{\mu} \]
    for $\lambda \vdash k$.
\end{prop}

Here is the special case of interest in this paper:

\begin{cor}
    For $0 \leq p \leq k/2$, we have
    \begin{equation} \label{eq:em}
        \sfe_{(k-p,p)} = \sum_{0 \leq q \leq p} \binom{k-2q}{p-q} \sfm_{2_k^q} \text{.}
    \end{equation}
    In the other direction, we have
    \begin{equation} \label{eq:me}
        \sfm_{2_k^q} = (-1)^q \sum_{0 \leq r \leq q} (-1)^r \sfe_{(k-r,r)} \left( \binom{k-q-r}{k-2q} + \binom{k-q-r-1}{k-2q} \right)
    \end{equation}
    for $0 \leq q \leq k/2-1$ and
    \begin{align*}
        \sfm_{2_k^{k/2}} &= \sfe_{(k/2,k/2)} + 2 \cdot (-1)^{k/2} \sum_{0 \leq r \leq k/2-1} (-1)^r \sfe_{(k-r,r)} \\
        &= (-1)^{k/2} \sum_{i+j=k} (-1)^i \sfe_i \sfe_j \text{.}
    \end{align*}
\end{cor}
\begin{proof}
    For (1), recall the interpretation of $K(\lambda,\mu)$ as the number of semistandard tableaux with shape $\lambda$ and weight $\mu$. For $\lambda=2_k^r$ and $\mu=2_k^q$, any such tableau must begin
    \[ \ytableausetup{boxsize=1.3em} \begin{ytableau}
        1 & 1 \\
        2 & 2 \\
        \none[\raisebox{-0.2em}{\vdots}] & \none[\raisebox{-0.2em}{\vdots}] \\
        q & q \\
        \, & \, \\
        \none[\raisebox{-0.2em}{\vdots}] & \none[\raisebox{-0.2em}{\vdots}] \\
        \, & \, \\
        \, \\
        \none[\raisebox{-0.2em}{\vdots}] \\
        \,
    \end{ytableau} \]
    so
    \[ K(2_k^r,2_k^q) = \dim(2_{k-2q}^{r-q}) = \frac{(k-2q)! (k-2p+1)}{(p-q)! (k-p-q+1)!} \]
    by the hook-length formula. On the other hand, a semistandard tableau of shape $(k-r,r)$ with weight $(k-p,p)$ must be of the form
    \[ \ytableausetup{boxsize=1.3em} \begin{ytableau}
        \, & \none[\cdots] & \, & \, & \none[\cdots] & \, & \, & \none[\cdots] & \, \\
        \, & \none[\cdots] & \, & 2 & \none[\cdots] & 2
    \end{ytableau} \]
    since the $2$s cannot go anywhere else if the other boxes are supposed to be filled with $1$s. So $K((k-r,r),(k-p,p)) = 1$ if $r \leq p$, otherwise it is $0$. Now, what remains is to show that
    \begin{equation} \label{eq:binomdiff}
        \sum_{q \leq r \leq p} \frac{(k-2q)! (k-2r+1)}{(r-q)! (k-r-q+1)!} = \binom{k-2q}{p-q} \text{.}
    \end{equation}
    To this end, observe that
    \begin{align*}
        \binom{k-2q}{r-q} - \binom{k-2q}{r-q-1} &= \frac{(k-2q)!}{(r-q)! (k-r-q)!} - \frac{(k-2q)!}{(r-q-1)! (k-r-q+1)!} \\
        &= \frac{(k-2q)! (k-r-q+1) - (k-2q)! (r-q)}{(r-q)! (k-r-q+1)!} \\
        &= \frac{(k-2q)! (k-2r+1)}{(r-q)! (k-r-q+1)!}
    \end{align*}
    so the only summand which is not cancelled out on the left-hand side of \cref{eq:binomdiff} is $\binom{k-2q}{p-q}$.
    
    For (2), one can refer to \cite{ER90} to find that
    \[ K^{-1}(2_k^q,2_k^s) = (-1)^{q-s} \binom{k-q-s}{k-2q} \]
    for $0 \leq s \leq q \leq k/2$, and
    \[ K^{-1}((k-r,r),(k-s,s)) = \begin{cases}
        1 & \text{if } s = r \\
        -1 & \text{if } s = r+1 \\
        0 & \text{otherwise}
    \end{cases} \]
    for $0 \leq r,s \leq k/2$.
\end{proof}

\subsection{Proof of (1) in \texorpdfstring{\cref{prop:LRdep}}{Proposition 4.7}}

The goal of this subsection is to show that
\begin{align*}
    &\quad \sum_{l=0}^k \frac{(-1)^l}{\binom{k}{l}} \sum_{\substack{S \subseteq [d] \\ | S | = k}} \sfe_{k-l}(A_S) \sfe_l(A_S) \\
    &= \frac{(k/2)!}{k!} \sum_{i+j=k} (-1)^i \frac{(d-i)! (d-j)!}{(d-k)! (d-k/2)!} \sfe_i(A) \sfe_j(A)
\end{align*}
when $k$ is even. Observe that
\begin{align*}
    \sum_{\substack{S \subseteq [d] \\ | S | = k}} \sfe_{k-l}(A_S) \sfe_l(A_S) &= \sum_{1 \leq s_1 < \cdots < s_k \leq d} \sum_{\substack{\bfi : [k] \to [k] \\ \bfi(1) < \cdots < \bfi(k-l) \\ \bfi(k-l+1) < \cdots < \bfi(k)}} a_{s_{\bfi(1)}} \cdots a_{s_{\bfi(k)}} \\
    &= \sum_{\substack{\bfi : [k] \to [k] \\ \bfi(1) < \cdots < \bfi(k-l) \\ \bfi(k-l+1) < \cdots < \bfi(k)}} \sum_{1 \leq s_1 < \cdots < s_k \leq d} a_{s_1}^{| \bfi^{-1}(1) |} \cdots a_{s_k}^{| \bfi^{-1}(k) |}
\end{align*}
and at a glance, the polynomial
\[ \sum_{1 \leq s_1 < \cdots < s_k \leq d} a_{s_1}^{| \bfi^{-1}(1) |} \cdots a_{s_k}^{| \bfi^{-1}(k) |} \]
calls to mind the monomial quasisymmetric functions, but
\[ (| \bfi^{-1}(1) |,\ldots,| \bfi^{-1}(k) |) \]
is not a valid index since quasisymmetric functions are supposed to be indexed by ordinary integer compositions, whose entries are all positive, whereas the index here must be allowed to have zero entries.

\begin{nota}
    For $I \in \WComp(m)$ with $\ell(I) = k$, write
    \[ \sfM_I(\bfx) := \sum_{s_1 < \cdots < s_k} x_{s_1}^{I_1} \cdots x_{s_k}^{I_k} \text{.} \]
\end{nota}

The above observation leads to the first piece of the proof:

\begin{lem}
    \label{lem:leftdepident}
    The expression \emph{(\ref{eq:leftdep})} is equal to
    \[ \sum_{q=0}^{k/2} \sfm_{2_k^q}(A) \left( \frac{\binom{d-(k-q)}{q}}{\binom{k-q}{q}} \sum_{l=q}^{k-q} \frac{(-1)^l}{\binom{k}{l}} \binom{k-q}{l} \binom{l}{q} \right) \text{.} \]
\end{lem}

To prove this, let us set up some more notation:

\begin{nota}
    \label{nota:SplitChains}
    Write
    \[ \C_m(k,l) := \{ \bfi : [k] \to [m] : \bfi(1) < \cdots < \bfi(k-l) \text{ and } \bfi(k-l+1) < \cdots < \bfi(k) \} \]
    for $0 \leq l \leq k$. For $\bfi \in \C_m(k,l)$, define a weak composition $I(\bfi) \in \WComp(m)$ by $I(\bfi) := (| \bfi^{-1}(1) |,\ldots,| \bfi^{-1}(k) |)$.
\end{nota}

\begin{lem}
    \label{lem:SplitChainsCount}
    Let $0 \leq q \leq \frac{k}{2}$. Then
    \begin{enumerate}
        \item $| \Orb(2^q,1^{k-2q},0^q) | = \binom{k}{q} \binom{k-q}{q}$;
        \item for $0 \leq l \leq k$, we have
            \begin{align*}
                &\quad | \{ \bfi \in \C_k(k,l) : I(\bfi) \in \Orb(2^q,1^{k-2q},0^q) \} | \\
                &= \begin{cases}
                    \binom{k}{l} \binom{k-l}{q} \binom{l}{q} & \text{if } q \leq l \leq k-q \\
                    0 & \text{otherwise}
                \end{cases} \text{.}
            \end{align*}
    \end{enumerate}
\end{lem}
\begin{proof}
    For (1), the distinct permutations of
    \[ (\underbrace{2,\ldots,2}_q,\underbrace{1,\ldots,1}_{k-2q},\underbrace{0,\ldots,0}_q) \]
    are determined by placing $q$ $2$s in $k$ available entries, then placing $q$ $0$s in the remaining $k-q$ available entries; the $k-2q$ $1$s are then forced into the remaining $k-2q$ entries. There are of course $\binom{k}{q} \binom{k-q}{q}$ ways of doing this.
    
    For (2), to build a $\bfi$ with $I(\bfi)$ a permutation of $(2^q,1^{k-2q},0^q)$, one may proceed as follows:
    \begin{itemize}
        \item start with a chain $\bfi(1) < \cdots < \bfi(k-l)$;
        \item choose $q$ values of the above, which will be duplicated;
        \item choose which of $\bfi(k-l+1),\ldots,\bfi(k)$ will be used for the duplication.
    \end{itemize}
    There are $\binom{k}{l}$ choices for the first, $\binom{k-l}{q}$ choices for the second, and $\binom{l}{q}$ choices for the third, hence the claim.
\end{proof}

\begin{proof}[Proof of \cref{lem:leftdepident}]
    With \cref{nota:SplitChains}, the expression (\ref{eq:leftdep}) is equal to
    \begin{align}
        &\quad \sum_{l=0}^k \frac{(-1)^l}{\binom{k}{l}} \sum_{\bfi \in \C_k(k,l)} \sfM_{I(\bfi)}(A) \nonumber \\
        &= \sum_{l=0}^k \frac{(-1)^l}{\binom{k}{l}} \sum_{q=0}^{k/2} \sum_{I \in \Orb(2^q,1^{k-2q},0^q)} | \{ \bfi \in \C_k(k,l) : I(\bfi) = I \} | \sfM_I(A) \nonumber \\
        &= \sum_{q=0}^{k/2} \sum_{I \in \Orb(2^q,1^{k-2q},0^q)} \left( \sum_{l=0}^k \frac{(-1)^l}{\binom{k}{l}} | \{ \bfi \in \C_k(k,l) : I(\bfi)=I \} | \right) \sfM_I(A) \label{eq:Ccard} \\
        &= \sum_{q=0}^{k/2} \sum_{I \in \Orb(2^q,1^{k-2q},0^q)} \sfM_I(A) \nonumber \\
        &\qquad \left( \frac{1}{\binom{k}{q} \binom{k-q}{q}} \sum_{l=q}^{k-q} \frac{(-1)^l}{\binom{k}{l}} \binom{k}{l} \binom{k-l}{q} \binom{l}{q} \right) \nonumber \tag{\cref{lem:SplitChainsCount}} \text{.}
    \end{align}
    It is easy to see that
    \[ \frac{1}{\binom{k}{q} \binom{k-q}{q}} \sum_{l=q}^{k-q} \frac{(-1)^l}{\binom{k}{l}} \binom{k}{l} \binom{k-l}{q} \binom{l}{q} = \frac{1}{\binom{k-q}{q}} \sum_{l=q}^{k-q} \frac{(-1)^l}{\binom{k}{l}} \binom{k-q}{l} \binom{l}{q} \]
    by pushing around some factorials, so the remaining task is to show that
    \[ \sum_{I \in \Orb(2^q,1^{k-2q},0^q)} \sfM_I(A) = \binom{d-(k-q)}{q} \sfm_{2_k^q}(A) \text{.} \]
    To this end, recall the definition
    \[ \sfm_{2_k^q}(A) = \sum_{J \in \Orb(2^q,1^{k-2q},0^{d-(k-q)})} a_1^{J_1} \cdots a_d^{J_d} \text{,} \]
    i.e. we add zeros to $2_k^q$ as ``padding" in case its length is less than $d$. On the other hand, we can write
    \begin{align*}
        &\quad \sum_{I \in \Orb(2^q,1^{k-2q},0^q)} \sfM_I(A) \\
        &= \sum_{I \in \Orb(2^q,1^{k-2q},0^q)} \sum_{1 \leq s_1 < \cdots < s_k \leq d} a_{s_1}^{I_1} \cdots a_{s_k}^{I_k} \\
        &= \sum_{I \in \Orb(2^q,1^{k-2q},0^q)} \sum_{1 \leq s_1 < \cdots < s_k \leq d} a_1^0 \cdots a_{s_1-1}^0 a_{s_1}^{I_1} a_{s_1+1}^0 \cdots a_{s_k-1}^0 a_{s_k}^{I_k} a_{s_k+1}^0 \cdots a_d^0
    \end{align*}
    so each summand is of the form $a_1^{J_1} \cdots a_d^{J_d}$ with
    \[ J = (0,\ldots,0,\underbrace{I_1}_{s_1},0,\ldots,0,\underbrace{I_k}_{s_k},0,\ldots,0) \in \Orb(2^q,1^{k-2q},0^{d-(k-q)}) \text{.} \]
    Every $J \in \Orb(2^q,1^{k-2q},0^{d-(k-q)})$ arises as such, in $\binom{d-(k-q)}{q}$ ways, since an element of the preimage is the same as a choice of $q$ $0$s to keep from the $d-(k-q)$ $0$s in $J$.
\end{proof}

For the remainder of the proof, we will require two identities of binomial coefficients, which can be found in e.g. \cite{G72}. In these identities, $y$ is a formal variable.

\begin{lem}[{\cite[4.8]{G72}}]
    \label{lem:binom}
    We have
    \[ \sum_{s=0}^{2n} (-1)^s \frac{\binom{2n}{s}}{\binom{2n+2y}{s+y}} = \frac{\binom{2n}{n}}{\binom{y+n}{n} \binom{2y+2n}{y+n}} \]
    for $n \geq 1$.
\end{lem}

\begin{lem}[Rothe-Hagen identity {\cite[3.146]{G72}}]
    \label{lem:rothehagen}
    We have
    \[ \sum_{s=0}^n \frac{n}{n+s} \binom{n+s}{s} \binom{y-s}{n-s} = \binom{n+y}{n} \]
    for $n \geq 1$.
\end{lem}

The right-hand side of \cref{lem:leftdepident} can be re-arranged as
\begin{align*}
    &\quad \sum_{q=0}^{k/2} \sfm_{2_k^q} \left( \frac{\binom{d-(k-q)}{q}}{\binom{k-q}{q}} \sum_{l=q}^{k-q} \frac{(-1)^l}{\binom{k}{l}} \binom{k-q}{l} \binom{l}{q} \right) \\
    &= \sum_{0 \leq q \leq k/2} \sfm_{2_k^q} \binom{d-k+q}{q} \left( \sum_{q \leq l \leq k-q} (-1)^l \frac{\binom{k-2q}{l-q}}{\binom{k}{l}} \right) \\
    &= \sum_{0 \leq q \leq k/2} \sfm_{2_k^q} \binom{d-k+q}{q} \left( (-1)^q \frac{\binom{k-2q}{k/2-q}}{\binom{k/2}{q} \binom{k}{k/2}} \right) \tag{\cref{lem:binom}} \\
    &= \frac{(k/2)!}{k!} \sum_{0 \leq q \leq k/2} (-1)^q \sfm_{2_k^q} (d-k+q)_q \frac{(k-2q)!}{(k/2-q)!}
\end{align*}
and one can apply \cref{eq:me}: the above is equal to
\begin{align}
    &\quad \frac{(k/2)!}{k!} \sum_{0 \leq r \leq k/2-1} (-1)^r \sfe_{(k-r,r)} \frac{(d-k+r)!}{(d-k)!} \nonumber \\
    &\qquad \sum_{r \leq q \leq k/2-1} (d-k+q) \cdots (d-k+r+1) \frac{(k-2r) (k-q-r-1)!}{(q-r)! (k/2-q)!} \nonumber \\
    &\quad + \frac{(k/2)!}{k!} (d-k/2)_{k/2} \sum_{i+j=k} (-1)^i \sfe_i \sfe_j \nonumber \\
    &= \frac{(k/2)!}{k!} \sum_{0 \leq r \leq k/2-1} (-1)^r \sfe_{(k-r,r)} \frac{(d-k+r)!}{(d-k)!} \nonumber \\
    &\qquad \sum_{r \leq q \leq k/2-1} \binom{d-k+q}{q-r} \frac{(k-2r) (k-q-r-1)!}{(k/2-q)!} \label{eq:prerothehagen} \\
    &\quad + \frac{(k/2)!}{k!} \frac{(d-k/2)!}{(d-k)!} \sum_{i+j=k} (-1)^i \sfe_i \sfe_j \text{.} \nonumber
\end{align}
By \cref{lem:rothehagen} with $n=k/2-r$ and $y = d-k/2$, the expression (\ref{eq:prerothehagen}) is equal to
\[ 2 \frac{(d-r)!}{(d-k/2)!} - 2 \frac{(d-k/2)!}{(d-k+r)!} \]
so the expression (\ref{eq:leftdep}) is equal to
\begin{align*}
    &\quad \frac{(k/2)!}{k!} \left( 2 \sum_{0 \leq r \leq k/2-1} (-1)^r \sfe_{(k-r,r)} \frac{(d-r)! (d-k+r)!}{(d-k)! (d-k/2)!} \right. \\
    &\qquad \left. - 2 \frac{(d-k/2)!}{(d-k)!} \sum_{0 \leq r \leq k/2-1} (-1)^r \sfe_{(k-r,r)} + \frac{(d-k/2)!}{(d-k)!} \sum_{i+j=k}(-1)^i \sfe_i \sfe_j \right) \\
    &= \frac{(k/2)!}{k!} \left( 2 \sum_{0 \leq r \leq k/2-1} (-1)^r \sfe_{(k-r,r)} \frac{(d-r)! (d-k+r)!}{(d-k)! (d-k/2)!} \right. \\
    &\qquad \left. + (-1)^{k/2} \frac{(d-k/2)!}{(d-k)!} \sfe_{(k/2,k/2)} \right) \\
    &= \frac{(k/2)!}{k!} \sum_{i+j=k} (-1)^i \frac{(d-i)! (d-j)!}{(d-k)! (d-k/2)!} \sfe_i \sfe_j
\end{align*}
hence the claim of (1) in \cref{prop:LRdep}.

\subsection{Proof of (2) in \texorpdfstring{\cref{prop:LRdep}}{Proposition 4.7}}

The remaining part of \cref{prop:LRdep} is the basis transition
\begin{align*}
    &\quad \frac{1}{k!} \sum_{0 \leq p \leq k/2} (-1)^p \frac{\dim(2_k^p)^2}{\sfs_{2_k^p}(1^d)} \sum_{0 \leq q \leq p} C_{2_k^p,2_k^q} q! (k-2q)! \sfm_{2_k^q}(B) \\
    &= k! \frac{d+1-k/2}{(d+1)! d!} \sum_{i+j=k} (-1)^i (d-i)! (d-j)! \sfe_i(B) \sfe_j(B)
\end{align*}
which is much more straightforward to prove than the previous one. The left-hand side is
\begin{align}
    &\quad \frac{1}{k!} \sum_{0 \leq q \leq p \leq k/2} (-1)^p \frac{\dim(2_k^p)^2}{\sfs_{2_k^p}(1^d)} \frac{p!}{(p-q)!} \binom{k-p+1}{q} q! (k-2q)! \sfm_{2_k^q} \tag{\cref{prop:Multiple}} \nonumber \\
    &= k! \sum_{0 \leq q \leq p \leq k/2} (-1)^p \frac{(d+1-p)! (d-(k-p))!}{(d+1)! d!} \frac{(k-2q)! (k-2p+1)}{(p-q)! (k-p-q+1)!} \sfm_{2_k^q} \tag{\cref{lem:Hooks}} \nonumber \\
    &= k! \sum_{0 \leq q \leq k/2} \sfm_{2_k^q} \nonumber \\
    &\qquad \sum_{q \leq p \leq k/2} (-1)^p \frac{(d+1-p)! (d-(k-p))!}{(d+1)! d!} \left( \binom{k-2q}{p-q} - \binom{k-2q}{p-q-1} \right) \label{eq:psum}
\end{align}
and with $Q_p(d) := \frac{(d+1-p)! (d-(k-p))!}{(d+1)! d!}$, the expression (\ref{eq:psum}) is equal to
\[ \sum_{q \leq p \leq k/2} (-1)^p \binom{k-2q}{p-q} (Q_p(d) + Q_{p+1}(d)) \]
where for the sake of notation we say $Q_{k/2+1}(d) = 0$. Then, for $q \leq p \leq k/2-1$, we have
\[ Q_p(d) + Q_{p+1}(d) = 2 \frac{d+1-k/2}{(d+1)! d!} (d-p)! (d-(k-p))! \]
and
\[ Q_{k/2}(d) = \frac{(d+1-k/2)! (d-k/2)!}{(d+1)! d!} = \frac{d+1-k/2}{(d+1)! d!} (d-k/2)! (d-k/2)! \text{.} \]
Putting this back into (\ref{eq:psum}), we get
\begin{align*}
    &\quad k! \sum_{0 \leq q \leq k/2} \sfm_{2_k^q} \sum_{q \leq p \leq k/2} (-1)^p \binom{k-2q}{p-q} (Q_p(d) + Q_{p+1}(d)) \\
    &= k! \frac{d+1-k/2}{(d+1)! d!} \sum_{0 \leq p \leq k/2} (-1)^p (d-p)! (d-(k-p))! \sum_{0 \leq q \leq p} \binom{k-2q}{p-q} \sfm_{2_k^q} \\
    &= k! \frac{d+1-k/2}{(d+1)! d!} \sum_{i+j=k} (-1)^i (d-i)! (d-j)! \sfe_i \sfe_j \tag{\cref{eq:em}}
\end{align*}
hence the claim of (2) in \cref{prop:LRdep}.

\appendix

\section{Young subgroups and permutation modules}
\label{sec:Young}

In this section we will prove \cref{prop:Multiple}, which amounts to the following:
\begin{propp}{prop:Multiple}
    Let $\lambda,\mu \vdash k$. Then
    \[ \sum_{\substack{\pi \in P(k) \\ t(\pi) = \mu}} \sum_{\substack{\tau \in S_k \\ \tau \leq \pi}} \rho^{\lambda}(\tau) = \frac{p_{\mu} | S_{\mu} | K(\lambda,\mu)}{\dim(\lambda)} \cdot 1 \]
    where $p_{\mu}$ is the number of partitions $\pi \in P(k)$ with $t(\pi) = \mu$.
\end{propp}

The idea is that the sum is averaging over each conjugate of $S_{\mu}$, and then adding up all the conjugate-subgroup-sums, yielding a central element of $\bC[S_k]$. Schur's lemma gives the scalar multiples, and then one can compute them as needed using the specifics of $S_k$.

An important part of the argument is clarified by working with finite groups in general:

\begin{nota}
    Fix a finite group $G$ and a subgroup $H$, and let
    \[ \{ g_j H g_j^{-1} : 1 \leq j \leq n \} \]
    be the distinct conjugates of $H$, writing $H_j := g_j H g_j^{-1}$. Fix a representation $\rho : G \to \GL(V)$ and write $\Res_H^G(\rho) = \bigoplus_{i=1}^m \rho_i$ where $\rho_1,\ldots,\rho_m$ are irreducible representations of $H$.
\end{nota}

The first general fact is that the restriction functor is invariant, up to natural isomorphism, under conjugation of subgroups:

\begin{lem}
    For $g \in G$, there is an isomorphism $\eta_g : \Res_H^G(\rho) \simeq \Res_{gHg^{-1}}^G(\rho)$.
\end{lem}
\begin{proof}
    Fix $g \in G$ and let $\rho : G \to \GL(V)$ be a representation of $G$. Define $\eta_g(\rho) : V \to V$ by $\eta_g(\rho) v = \rho(g) v \rho(g^{-1})$ for $v \in V$, which is an isomorphism of vector spaces. Moreover, for $h \in H$,
    \begin{align*}
        \rho(g h g^{-1}) \eta_g(\rho) v &= \rho(g h g^{-1}) \rho(g) v \rho(g^{-1}) \\
        &= \rho(g h) v \rho(g^{-1}) \\
        &= \rho(g) \rho(h) v \rho(g^{-1}) \\
        &= \eta_g(\rho) \rho(h) v
    \end{align*}
    for $v \in V$, so $\eta_g(\rho)$ intertwines $\Res_H^G(\rho)$ and $\Res_{g H g^{-1}}(\rho)$.
\end{proof}

The second general fact is that averaging a representation over a subgroup yields a projection which encodes the occurrence of the trivial representation in the restriction:

\begin{lem}
    \label{lem:ResMult}
    In the block-matrix decomposition with respect to $\bigoplus_{i=1}^m V_i$,
    \[ \frac{1}{| H |} \sum_{h \in H} \rho(h) = \begin{pmatrix}
        \delta_{\rho_1=\triv} \cdot 1 & \multicolumn{2}{c}{\raisebox{-1ex}{\large $0$}} \\
        & \ddots & \\
        \multicolumn{2}{c}{\raisebox{+1ex}{\large $0$}} & \delta_{\rho_m=\triv} \cdot 1
    \end{pmatrix} \text{.} \]
\end{lem}
\begin{proof}
    Clearly $\sum_{h \in H} \rho(h) \in \End_H(\rho)$, so in the block-matrix decomposition
    \[ \sum_{h \in H} \rho(h) = \begin{pmatrix}
        T_{11} & \cdots & T_{1m} \\
        \vdots & \ddots & \vdots \\
        T_{m1} & \cdots & T_{mm}
    \end{pmatrix} \]
    we have $T_{ij} \in \Hom_H(\rho_j,\rho_i)$, and then by Schur's lemma we have
    \[ \sum_{h \in H} \rho(h) = \begin{pmatrix}
        t_1 \cdot 1 & \multicolumn{2}{c}{\raisebox{-1ex}{\large $0$}} \\
        & \ddots & \\
        \multicolumn{2}{c}{\raisebox{1ex}{\large $0$}} & t_m \cdot 1
    \end{pmatrix} \]
    for some scalars $t_1,\ldots,t_m$. Write $\chi_i$ for the character of $\rho_i$, so
    \[ \la 1,\chi_i \ra = \frac{1}{| H |} \sum_{h \in H} \Tr(\rho_i(h)) = \frac{1}{| H |} \Tr(t_i \cdot 1) = \frac{1}{| H |} t_i \dim(\rho_i) \]
    and by the orthogonality relations we have $t_i = \delta_{\rho_i=\triv} \cdot | H |$.
\end{proof}

The final general fact combines the previous two:

\begin{lem}
    \label{lem:Central}
    The element
    \[ \sum_{j=1}^n \sum_{h \in H_j} h \in \bC[G] \]
    of the group algebra is central. If $\rho$ is irreducible, then
    \[ \sum_{j=1}^n \sum_{h \in H_j} \rho(h) = \frac{n | H | \mult(\triv,\Res_H^G(\rho))}{\dim(\rho)} \cdot 1 \text{.} \]
\end{lem}
\begin{proof}
    The set $\{ H_1,\ldots,H_n \}$ is permuted by elements of $G$ acting by conjugation, so
    \begin{align*}
        g \left( \sum_{j=1}^n \sum_{h \in H_j} h \right) g^{-1} &= \sum_{j=1}^n \sum_{h \in H_j} g h g^{-1} \\
        &= \sum_{j=1}^n \sum_{h \in g H_j g^{-1}} h \\
        &= \sum_{j=1}^n \sum_{h \in H_j} h
    \end{align*}
    for $g \in G$, which is the first claim. If $\rho$ is irreducible, then by Schur's lemma,
    \[ \sum_{j=1}^n \sum_{h \in H_j} h = t \cdot 1 \]
    for some $t \in \bC$. To find $t$, observe that
    \begin{align*}
        \dim(\rho) t &= \Tr \left( \sum_{j=1}^n \sum_{h \in H_j} \rho(h) \right) \\
        &= \sum_{j=1}^n \Tr \left( \sum_{h \in H_j} \rho(h) \right) \\
        &= \sum_{j=1}^n | H_j | \mult(\triv,\Res_{H_j}^G(\rho)) \\
        &= n | H | \mult(\triv,\Res_H^G(\rho))
    \end{align*}
    so the claim follows.
\end{proof}

Finally, let us specialize to the symmetric group $S_k$. The key point of this section, in relation to the problem considered in this paper, is that the constraint $\tau \leq \pi$ in the sum
\[ \sum_{\substack{\tau \in S_k \\ \tau \leq \pi}} \chi^{\lambda}(\sigma \tau) \]
is actually carving out a well-known subgroup of $S_k$:

\begin{nota}[Young subgroup conjugates]
    For $\pi \in P(k)$, write $S_{\pi}$ for the subgroup of $S_k$ consisting of the permutations for which the blocks of $\pi$ are invariant. Clearly, if $\pi = \{ V_1,\ldots,V_m \}$, then
    \[ S_{\pi} \simeq S_{| V_1 |} \times \cdots \times S_{| V_m |} \]
    and the right-hand side is the Young subgroup corresponding to the composition $(| V_1 |,\ldots,| V_m |)$ of $k$, but the notation $S_{\pi}$ retains some more information about the blocks of $\pi$.
\end{nota}

The main result needed here which is particular to the symmetric groups is sometimes called \emph{Young's rule}:

\begin{thm}
    \label{thm:FrobeniusYoung}
    For $\lambda,\mu \vdash k$, the multiplicity of $V^{\lambda}$ in the permutation module $\Ind_{S_{\mu}}^{S_k}(\triv)$ is the Kostka number $K(\lambda,\mu)$.
\end{thm}

\begin{proof}[Proof of \cref{prop:Multiple}]
    In light of \cref{lem:Central}, the remaining tasks are to count
    \begin{enumerate}
        \item $\{ \pi \in P(k) : t(\pi) = 2_k^q \}$,
        \item the order of $S_{2_k^q}$,
        \item the multiplicity of the trivial representation in $\Res_{S_{\pi}}^{S_k}(V^{\lambda})$, and
        \item the dimension of $V^{2_k^p}$;
    \end{enumerate}
    in particular, we want the multiplicity in (3) to be $0$ whenever $\mu \not\trianglelefteq \lambda$. There is a well-known formula for (1), reproduced in e.g. \cite[Lemma 2.4]{AP18}, and the case of $2_k^q$ comes out as $\frac{k!}{2^q q! (k-2q)!}$. For (2), we already know $S_{2_k^q} \simeq \bfZ_2^q$ which has order $2^q$. For (3) and (4), we appeal to \cref{thm:FrobeniusYoung} and \cref{lem:Hooks} respectively.
\end{proof}

\section*{Acknowledgements}

The author wishes to thank Alexandru Nica and Daniel Perales for their helpful feedback and encouragement at various stages of this project.

\end{document}